\documentclass{article}

\usepackage{arxiv}

\usepackage[utf8]{inputenc} 
\usepackage[T1]{fontenc}    
\usepackage{hyperref}       
\usepackage{url}            
\usepackage{booktabs}       
\usepackage{amsfonts}       
\usepackage{nicefrac}       
\usepackage{microtype}      
\usepackage{lipsum}		
\usepackage{graphicx}
\usepackage{natbib}
\usepackage{doi}

\usepackage{amsmath}
\usepackage{amssymb}
\usepackage{amsfonts}
\usepackage{bm}

\usepackage{amsthm}
\usepackage{overpic}
\usepackage{lipsum}
\usepackage{enumitem}
\usepackage{cases}
\usepackage{subcaption}
\usepackage{graphicx}

\DeclareMathOperator*{\argmax}{arg\,max}
\DeclareMathOperator*{\argmin}{arg\,min}

\usepackage{ifthen}
\usepackage{graphicx,epstopdf,amssymb,amsthm}
\usepackage{ulem,color}
\usepackage{epstopdf}

\usepackage[utf8]{inputenc}
\usepackage[english]{babel}
\usepackage{caption}
\usepackage{algorithm, algorithmicx}
\usepackage{algpseudocode}

\DeclareCaptionLabelFormat{lc}{\MakeLowercase{#1}~#2}
{
  \theoremstyle{plain}
  \newtheorem{assumption}{Assumption}
}

\newtheorem{theorem}{Theorem}
\newtheorem{corollary}{Corollary}

\newtheorem{lemma}{Lemma}

\newtheorem{definition}{Definition}

\newtheorem{example}{Example}[]
\newcommand{\edit}[1]{\textcolor{black}{#1}}

\algnewcommand{\Inputs}[1]{%
  \State \textbf{Inputs:}
  \Statex \hspace*{\algorithmicindent}\parbox[t]{.8\linewidth}{\raggedright #1}
}
\algnewcommand{\Initialize}[1]{%
  \State \textbf{Initialize:}
  \Statex \hspace*{\algorithmicindent}\parbox[t]{.8\linewidth}{\raggedright #1}
}

\title{Collaborative Safety-Critical Control for Dynamically Coupled Networked Systems}

\date{} 					

\author{ \href{https://orcid.org/0000-0002-2489-4411}{\includegraphics[scale=0.06]{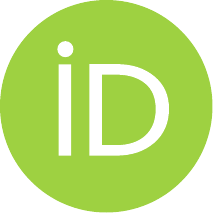}\hspace{1mm}Brooks A. Butler}\\
	Elmore Family School of Electrical and Computer Engineering\\
	Purdue University\\
	\texttt{brooksbutler@purdue.edu} \\
        \And
	\href{https://orcid.org/0000-0002-4095-7320}{\includegraphics[scale=0.06]{orcid.pdf}\hspace{1mm}Philip E. Par\'{e}}\thanks{This work was partially funded by Purdue’s Elmore Center for Uncrewed Aircraft Systems and the National Science Foundation, grant NSF-ECCS \#2238388.}\\
	Elmore Family School of Electrical and Computer Engineering\\
	Purdue University\\
	\texttt{philpare@purdue.edu} \\
}

\renewcommand{\headeright}{}
\renewcommand{\undertitle}{}
\renewcommand{\shorttitle}{Collaborative Safety-Critical Control for Dynamically Coupled Networked Systems}

\hypersetup{
pdftitle={Collaborative Safety-Critical Control for Dynamically Coupled Networked Systems},
pdfsubject={q-bio.NC, q-bio.QM},
pdfauthor={Brooks A. Butler, Philip E. Par\'{e}},
pdfkeywords={Network analysis and control, Safety-critical control, Cooperative control, Constrained control},
}

\begin{document}
\maketitle

\begin{abstract}
As modern systems become ever more connected with complex dynamic coupling relationships, developing safe control methods becomes paramount. In this paper, we discuss the relationship of node-level safety definitions for individual agents with local neighborhood dynamics. We define a 
collaborative control barrier function (CCBF) and provide conditions under which sets defined by these functions will be forward invariant. We use collaborative node-level control barrier functions to construct a novel \edit{decentralized} algorithm for the safe control of collaborating network agents and provide conditions under which the algorithm is guaranteed to converge to a viable set of safe control actions for all agents. 
We illustrate these results on a networked susceptible-infected-susceptible (SIS) model.
\end{abstract}


\section{Introduction}

Networked dynamic systems have become ubiquitous in modern society, underpinning critical infrastructure, transportation, communication, and \edit{other} coupled processes. In many applications, effective and safe operation in networked systems is crucial, as disruptions in these interconnected systems can potentially have far-reaching societal and economic consequences. 
One great challenge in effectively controlling such networked systems is the rate at which the complexity scales with each additional node, where the computational complexity for even relatively simple networked models can become exponentially intractable as these networks grow (i.e., as the number of nodes increases). To overcome this scaling challenge, decentralized controllers are developed to scale with network growth. \edit{Decentralized} control laws become especially important when nodes represent independent agents with individual goals.
Additionally, when network agents are considered to be independent actors with individual objectives, cooperative control schemes allow for coordination between networked agents via active communication \citep{li2017cooperative,wang2017cooperative,yu2017distributed}, further emphasizing the importance of a \edit{decentralized} control law in such systems.

One common strategy used to tackle the high dimensionality and complexity of networked systems is to break down these potentially large systems into smaller, and therefore more readily solvable, components and provide methods for composing a solution for the entire system. This strategy has been applied in the compositional construction of barrier functions for networked systems \citep{nejati2022compositional, jahanshahi2022compositional, jagtap2020compositional} where systems are interconnected via input-output connections. Oftentimes, the input of other systems to an interconnected component is treated as a bounded disturbance, \edit{leveraging} small-gain theory when composing multiple barrier certificates~\citep{anand2022small}. Other work on composing barrier functions, although not explicitly for networked systems, uses geometric methods to compose multiple safety constraints on second-order systems \citep{breeden2022compositions}, which utilizes principles of higher-order barrier functions \citep{xiao2021high}. Work has also been presented on the composition of neural certificates for networked systems \citep{zhang2023compositional}, which uses deep learning to train scalable models with networked connections based on principles of control Lyapunov functions combined with a small-gain assumption.

Another method for computing safety certificates for systems involves the use of signal temporal logic (STL) \citep{donze2013signal, raman2014model, deshmukh2017robust, raman2015reactive, sadigh2016safe} which provides a natural syntax for ensuring that a dynamic system meets certain safety requirements for all time. These methods have been applied to networked systems with coupled dynamics to formulate both centralized and decentralized controllers under STL tasks \citep{lindemann2020barrier,lindemann2019control,lindemann2019decentralized} which can be applied to time-varying barrier functions. Additionally, work has been done using assume-guarantee contracts and parameterized signal temporal logic (pSTL) as an analogy for barrier certificates in networked systems with applications to power systems \citep{chen2020safety,chen2019compositional}. In these works, input from neighboring nodes in the network is treated as a bounded disturbance, which in the case of \citep{chen2020safety}, these disturbances are considered to be scalar and summable. Many of these works often focus on the definition of safety for the network as a whole and use compositional methods to efficiently compute centralized safety certificates \citep{song2022generalization}. There is also recent work that uses machine learning to train distributed controllers for scalable networked systems via graph neural networks \citep{zhou2020graph,wu2020comprehensive,scarselli2008graph} 
to implement safe multi-agent control \citep{fan2023graphCBF} that can be tested by neighbors without knowledge of the control inputs.

The field of cooperative control for multi-agent systems provides a rich body of literature that examines scenarios where agents may share information over a communication network \citep{li2017cooperative,wang2017cooperative,yu2017distributed}. In such formulations, agents typically share and receive information via either direct communication or broadcast that enables cooperative control adjustments to be made \citep{huang2010adaptive,li2021robust,qu2012analytic}. However, in many formulations of cooperative control, a common assumption is that agent first-order dynamics are independent of each other, thus the networked element is mainly facilitated via virtual communication. Contrasting in this work, we wish to include the class of explicitly coupled systems and leverage any knowledge of the networked dynamic structure in the formulation of safety requests.

While the current literature has explored some approaches for \edit{decentralized} safety-critical control of networked dynamic systems, the use of explicit coupled network dynamics for safety-critical control of individual network agents and the implementation of active collaboration between coupled agents to achieve independent safety goals is, to the best of our knowledge, an open problem. Therefore, \edit{in this work, we} 
\edit{leverage the knowledge of the networked dynamics to define a decentralized, collaborative safety condition 
and use this safety condition to construct an algorithm 
that facilitates the individual safety of each node while simultaneously satisfying neighbors' safety needs.}  

\section{Preliminaries} \label{sec:preliminaries}
In this section, we define the notation to be used in this paper, provide a general definition for the class of networked dynamic systems discussed in this work, and discuss safety definitions for networked systems including node-level definitions of safety and their relationship to network dynamics.

\subsection{Notation}
\noindent 
Let $\text{Int} \mathcal{C}$, $\partial \mathcal{C}$, $|\mathcal{C}|$ denote the interior, boundary, and cardinality of the set $\mathcal{C}$, respectively. $\mathbb{R}$ and $\mathbb{N}$ are the set of real numbers and positive integers, respectively. Let $D^r$ denote the set of functions $r$-times continuously differentiable in all arguments, and $\mathcal{K}$ the set of class-$\mathcal{K}$ functions. We define $[n] \subset \mathbb{N}$ to be a set of indices $\{1, 2, \dots, n\}$.
We define the Lie derivative of the function $h:\mathbb{R}^N \rightarrow \mathbb{R}$ with respect to the vector field generated by $f:\mathbb{R}^N \rightarrow \mathbb{R}^N$ as
\begin{equation}
    \mathcal{L}_f h(x) = \frac{\partial h(x)}{\partial x} f(x).
\end{equation}
We define higher-order Lie derivatives with respect to the same vector field $f$ with a recursive formula \citep{robenack2008computation}, where $k>1$, as
\begin{equation}
    \mathcal{L}^k_f h(x) = \frac{\partial \mathcal{L}^{k-1}_f h(x)}{\partial x} f(x).
\end{equation}

\subsection{Networked Dynamic System Model} \label{sec:sub:networked_model}
We define a networked system using a graph $\mathcal{G} = (\mathcal{V}, \mathcal{E})$, where $\mathcal{V}$ is the set of $n = \vert \mathcal{V} \vert$ nodes, $ \mathcal{E} \subseteq \mathcal{V}\times \mathcal{V} $ is the set of edges. Let $\mathcal{N}_i^+$ be the set of all neighbors with an incoming connection to node $i$, where 
\begin{equation}
    \mathcal{N}_i^+ = \{j \in [n]\setminus \{i\}: (i,j) \in \mathcal{E} \}.
\end{equation}
Similarly, all nodes with an outgoing connection from $i$ to $j$ are given by
\begin{equation}
    \mathcal{N}_i^- = \{j \in [n]\setminus \{i\}: (j,i) \in \mathcal{E} \},
\end{equation}
with the complete set of neighboring nodes given by
\begin{equation}
    \mathcal{N}_i = \mathcal{N}_i^+ \cup \mathcal{N}_i^-.
\end{equation}

\noindent Further, we define the state vector for each node as $x_i \in \mathbb{R}^{N_i}$, with $N = \sum_{i \in [n]} N_i$ being the state dimension of the entire system, $N_i^+ = \sum_{j \in \mathcal{N}_i^+} N_j$ the combined dimension of incoming neighbor states, and $x_{\mathcal{N}_i^+} \in \mathbb{R}^{N_i^+}$ denoting the combined state vector of all incoming neighbors.  Then, for each node~$i \in [n]$, we can describe its state dynamics, which are nonlinear, time-invariant, \edit{and control-affine}, as
\begin{equation} \label{eq:sys_dyn_net}
    \dot{x}_i = f_i(x_i, x_{\mathcal{N}_i^+}) + g_i(x_i) u_i,
\end{equation}
where $f_i:\mathbb{R}^{N_i + N_i^+} \rightarrow \mathbb{R}^{N_i}$ and $g_i:  \mathbb{R}^{N_i} \rightarrow \mathbb{R}^{N_i} \times \mathbb{R}^{M_i}$ locally Lipschitz for all $i \in [n]$, and $u_i \in \mathcal{U}_i \subset \mathbb{R}^{M_i}$. Note that in our formulation of a networked dynamic system, we include coupling effects and networked structure as an integral part of our control-free model dynamics $f_i$, rather than treating coupling connections as disturbances or noise. In this sense, we aim to exploit any information the network structure provides in our control design. \edit{For notational compactness, given a node~$i \in [n]$, we collect the 1-hop neighborhood state as $\mathbf{x}_i = (x_i, x_{\mathcal{N}_i^+})$ and the 2-hop neighborhood state as $\mathbf{x}_i^+ = (x_i, x_{\mathcal{N}_i^+}, x_{\mathcal{N}_j^+}~\forall j \in \mathcal{N}_i^+)$.} 

\subsection{Safety Definitions} \label{sec:sub:safety_def}
While the language and syntax of set invariance and barrier functions provide a mathematically succinct grammar for describing theoretical safety, defining safety in practice requires careful delineation between desired safety goals and actual viable safety in the defined state space with respect to the system dynamics. Additionally, when discussing definitions of safety for networked dynamic systems, one must consider the possibly differing definitions of safety for each node in the network. These \textit{node-level} safety definitions become even more relevant when considering networked models where nodes may be viewed as independent agents working to achieve individual or node-level goals. Thus, in the context of a networked model defined in Section~\ref{sec:sub:networked_model}, we define a node-level safety constraint for node~$i \in [n]$ with the set

\begin{equation}\label{eq:safe_set_i_fullstate}
    \begin{aligned}
        \mathcal{C}_i &= \left\{ \edit{x_i} \in \mathbb{R}^{N_i} : h_i(x_i) \geq 0 \right\},
    \end{aligned}
\end{equation}
where $h_i \in D^r, r\geq 1$ and $h_i:\mathbb{R}^{N_i} \rightarrow \mathbb{R}$ is a function whose zero-super-level set defines the region which node $i\in [n]$ considers to be safe (i.e., if $h_i(x_i)<0$, then node $i$ is no longer safe). 
We define the safety constraints for the entire networked system as
\begin{equation}\label{eq:safe_set_net_fullstate}
    \begin{aligned} 
        \mathcal{C} &= \edit{\mathcal{C}_1 \times \cdots \times \mathcal{C}_n}. \;
    \end{aligned}
\end{equation}

\noindent
Given the definitions of these safety constraints, we define the viable safe regions \citep{breeden2022compositions} for each node as follows.
\begin{definition}
A set $\mathcal{S}_i \subseteq \mathcal{C}_i$ is called a \textit{node-level viability domain} for node $i$ if for every point $x_i(t_0) \in \mathcal{S}_i$ there exist a control signal $u_i(t) \in \mathcal{U}_i$ 
with $t \in \mathcal{T}$ such that the trajectory $x_i(\cdot)$ of \eqref{eq:sys_dyn_net} satisfies $x_i(t) \in \mathcal{S}_i$ for all $t \in \mathcal{T}$.
\end{definition}

We define a similar notion for the entire networked system, where we compose the dynamics for the entire system as
\begin{equation} \label{eq:full_sys_dyn_net}
    \dot{x} = f(x) + g(x)u
\end{equation}
where $x \in \mathbb{R}^N$, $f:\mathbb{R}^N \rightarrow \mathbb{R}^N$, $g:\mathbb{R}^N \times \mathbb{R}^M \rightarrow \mathbb{R}^N$,  $M = \sum_{i \in [n]} M_i$,  and $u \in \mathcal{U} \subset \mathbb{R}^M$ with 
\begin{equation*}
    \mathcal{U} = \mathcal{U}_1 \times \cdots \times \mathcal{U}_n,
\end{equation*}
which contains all the control actions taken across the network.
\begin{definition}
 A set $\mathcal{S} \subseteq \mathcal{C}$ is called a \textit{network viability domain} if for every point $x(t_0) \in \mathcal{S}$ there exist control signals $u(t) \in \mathcal{U}$ with $t \in \mathcal{T}$ such that the trajectory $x(\cdot)$ of \eqref{eq:full_sys_dyn_net} satisfies $x(t) \in \mathcal{S}$ for all $t \in \mathcal{T}$.   
\end{definition}

Note that viability domains may be defined by an \textit{implicit barrier function} \citep{gurriet2020scalable} with respect to a defined backup set $\mathcal{S}^b \subseteq \mathcal{S}$ which the system can safely return to within a given time horizon. 

\subsection{\edit{Problem Statement}} \label{sec:problem_statement}

\edit{
With system dynamics and safety definitions defined, we are prepared to formalize the problem we aim to solve in this work as follows: Given a networked system with node-level dynamics defined by \eqref{eq:sys_dyn_net}, individual safety requirements defined by \eqref{eq:safe_set_i_fullstate}, and control constraints $\mathcal{U}_i \subset \mathbb{R}^{M_i}$, our objective is to design a decentralized control law $u_i(t)$, for all $i\in [n]$, that enforces \textit{forward invariance} on $\mathcal{C}_i$.
}

\section{Safety in Networked Dynamic Systems} \label{sec:safe_w_local_control}

One challenge that comes from the definition of node-level safety constraints as defined in Section~\ref{sec:sub:safety_def}, which depend only on the state of the given node, is that when evaluating the derivative of $h_i(x_i)$ we must incorporate the network influence of its neighbors $j \in \mathcal{N}_i^+$, where we define the derivative of $h_i$ as
\begin{equation*}
    \dot{h}_i(\edit{\mathbf{x}_i}, u_i) = \mathcal{L}_{f_i} h_i(\edit{\mathbf{x}_i}) + \mathcal{L}_{g_i} h_i(x_i) u_i. 
\end{equation*}
\noindent
Notice that in this case, we can define the node-level constraint function as the mapping $h_i:\mathbb{R}^{N_i} \rightarrow \mathbb{R}$; however, the dynamics of node $i$, whose dynamics are a function of all neighboring nodes in $\mathcal{N}_i^+$, require that $\dot{h}_i:\mathbb{R}^{N_i+N_i^+} \rightarrow \mathbb{R}$. 

Therefore, even though we define $h_i$ with respect to only $x_i \in \mathbb{R}^{N_i}$, when computing the Lie derivative we must compute the Jacobian with respect to all network states $\frac{\partial h_i(x_i)}{\partial x}$, which naturally zeros out all states of nodes $j \notin \mathcal{N}_i^+ \cup \{i\}$ when taking only the first derivative, but is relevant when we must consider higher-order derivatives of the constraint function $h_i$. 

Recall the definition of \textit{high-order barrier functions} (HOBF) \citep{xiao2021high}, where a series of functions are defined in the following form
\begin{equation} \label{eq:HO_funcs}
    \begin{aligned}
        \psi_i^0(x) &:= h_i(x) \\
        \psi_i^1(x) &:= \dot{\psi}_i^0(x) + \alpha_i^1(\psi_i^0(x)) \\
        & \vdots \\
        \psi_i^k(x) &:= \dot{\psi}_i^{k-1}(x) + \alpha_i^k(\psi_i^{k-1}(x)),
    \end{aligned}
\end{equation}
where $\alpha_i^1(\cdot),\alpha_i^1(\cdot), \dots, \alpha_i^k(\cdot)$ denote class-$\mathcal{K}$ functions of their argument. These functions provide definitions for the corresponding series of sets
\begin{equation} \label{eq:HO_sets}
    \begin{aligned}
        \mathcal{C}_i^1 &:= \{ x \in \mathbb{R}^N: \psi_i^0(x) \geq 0 \} \\
        \mathcal{C}_i^2 &:= \{ x \in \mathbb{R}^N: \psi_i^1(x) \geq 0 \} \\
        & \vdots \\
        \mathcal{C}_i^k &:= \{ x \in \mathbb{R}^N: \psi_i^{k-1}(x) \geq 0 \},
    \end{aligned}
\end{equation}
which yield the following definition.
\begin{definition}
    Let $\mathcal{C}_i^1, \mathcal{C}_i^2, \dots, \mathcal{C}_i^k$ be defined by \eqref{eq:HO_funcs} and \eqref{eq:HO_sets}. We have that $h_i$ is a \textit{\edit{$k^{th}$-order} node-level barrier function} (NBF) for node $i\in [n]$ if $h_i \in C^k$ and there exist differentiable class-$\mathcal{K}$ functions $\alpha_i^1,\alpha_i^2,\dots, \alpha_i^k$ such that $\psi_i^k(x) \geq 0$ for all $x \in \bigcap_{r=1}^k \mathcal{C}_i^r$. 
\end{definition}

This definition leads naturally to the following lemma (which is a direct result of Theorem 4 in \citep{xiao2021high}).
\begin{lemma} \label{lem:NBF}
    If $h_i$ is an NBF, the set $\bigcap_{r=1}^k \mathcal{C}_i^r$ is forward invariant.
\end{lemma}

In this sense, under Lemma~\ref{lem:NBF} we may consider
\begin{equation} \label{eq:node_lvl_viab}
    \mathcal{S}_i = \bigcap_{r=1}^k \mathcal{C}_i^r \subseteq \mathcal{C}_i
\end{equation}
to be a node-level viability domain of $\mathcal{C}_i$ with respect to the $(k-1)$-hop neighborhood dynamics of node~$i \in [n]$. 

\edit{Thus, to analyze the effect of the 1-hop neighborhood dynamics on the safety of node~$i$, we must compute the second-order derivative of $h_i$} with respect to the network dynamics defined by \eqref{eq:sys_dyn_net}, \edit{which can be expressed as}

\begin{equation} \label{eq:sec_der_h_i_control_aff}
    \begin{aligned}
        \ddot{h}_i(\edit{\mathbf{x}_i^+,u_i, u_{\mathcal{N}_i^+}, \dot{u}_i}) &= \sum_{j \in \mathcal{N}_i^+} \big[ \mathcal{L}_{f_j} \mathcal{L}_{f_i}  h_i(\edit{\mathbf{x}_i^+}) +  \mathcal{L}_{g_j} \mathcal{L}_{f_i}  h_i(\edit{\mathbf{x}_i}) u_j \big] \\
         &
         \quad 
         + \mathcal{L}^2_{f_i}  h_i(\edit{\mathbf{x}_i}) +  u_i^\top \mathcal{L}^2_{g_i} h_i(x_i) u_i \edit{+ \mathcal{L}_{g_i}h_i(x_i)\dot{u}_i} \\ 
         &
         \quad 
         + \big(\mathcal{L}_{f_i} \mathcal{L}_{g_i}h_i(\edit{\mathbf{x}_i})^\top +  \mathcal{L}_{g_i} \mathcal{L}_{f_i}h_i(\edit{\mathbf{x}_i})\big)u_i.
    \end{aligned}
\end{equation}

\noindent
Notice that the dynamics of each neighbor $j \in \mathcal{N}_i^+$ 
\edit{
and $\dot{u}_i$ appear in the second-order differential expression of $h_i$. To assist in our analysis of the high-order dynamics of $h_i$, we make the following assumption.
}
\edit{
\begin{assumption} \label{assume:u_dot_func_u}
    For a given node~$i\in [n]$, let $\dot{u}_i := d(u_i)$, where $d(u_i): \mathbb{R}^{M_i} \rightarrow \mathbb{R}^{M_i}$ is locally Lipschitz.
\end{assumption}
}
\edit{
While obtaining a closed-form solution for $\dot{u}_i$ may be challenging in some applications, in practice, $d(u_i)$ may be approximated using discrete-time methods. 
}
For the purposes of implementing a \edit{decentralized} control law, we need only consider the 1-hop neighborhood dynamics of any given node, which may then be applied recursively across the entire network. Additionally, since control is applied locally at each node, we have that the combined control for $h_i$ is of \textit{mixed relative degree} due to the control input of neighbors appearing in the second-order derivative of $h_i$. Given this structure in the 1-hop neighborhood controlled dynamics, we construct a similar series of functions to \eqref{eq:HO_funcs} for the second-order system for the 1-hop neighborhood of node~$i \in [n]$, \edit{under Assumption~\ref{assume:u_dot_func_u}, as}
\begin{equation} \label{eq:MO_funcs}
    \begin{aligned}
        \psi_i^0(x_i) &:= h_i(x_i) \\
        \psi_i^1(\edit{\mathbf{x}_i},u_i) &:= \dot{\psi}_i^0(\edit{\mathbf{x}_i},u_i) + \eta_i(\psi_i^0(x_i)) \\
        \psi_i^2(\edit{\mathbf{x}_i^+},u_i,u_{\mathcal{N}_i^+}) &:= \dot{\psi}_i^1(\edit{\mathbf{x}_i^+},u_i,u_{\mathcal{N}_i^+}) + \kappa_i(\psi_i^1(\edit{\mathbf{x}_i}, u_i)),
    \end{aligned}
\end{equation}
where $\eta_i,\kappa_i$ are class-$\mathcal{K}$ functions. 
We can also express $\psi_i^2(\edit{\mathbf{x}_i^+},u_i,u_{\mathcal{N}_i^+})$ in terms of $h_i$ as
\begin{equation} \label{eq:psi2_terms_h}
    \begin{aligned}
        \psi_i^2(\edit{\mathbf{x}_i^+},u_i,u_{\mathcal{N}_i^+}) &= \ddot{h}_i(\edit{\mathbf{x}_i^+},u_i,u_{\mathcal{N}_i^+}) + \dot{\eta}_i(h_i(x_i),\edit{\mathbf{x}_i},u_i) \\
        & \quad + \kappa_i\big(\dot{h}_i(\edit{\mathbf{x}_i},u_i) + \eta_i(h_i(x_i)) \big).
    \end{aligned}
\end{equation}
Further, we may rewrite \eqref{eq:psi2_terms_h} by collecting the terms independent of $u_j$ as

\begin{equation} \label{eq:psi2_grouped_terms}
     \psi_i^2(\edit{\mathbf{x}_i^+},u_i,u_{\mathcal{N}_i^+}) = \sum_{j \in \mathcal{N}_i^+} a_{ij}(\edit{\mathbf{x}_i}) u_j  +  c_i(\edit{\mathbf{x}_i^+},u_i),
\end{equation}
where
\begin{equation} \label{eq:a_ij}
    a_{ij}(\edit{\mathbf{x}_i}) = \mathcal{L}_{g_j} \mathcal{L}_{f_i}  h_i(\edit{\mathbf{x}_i})
\end{equation}
and
\begin{equation} \label{eq:capability_node_i}
    \begin{aligned}
        c_i(x,u_i) &= \sum_{j \in \mathcal{N}_i^+} \mathcal{L}_{f_j} \mathcal{L}_{f_i}  h_i(\edit{\mathbf{x}_i^+}) + \mathcal{L}^2_{f_i} h_i(\edit{\mathbf{x}_i}) + u_i^\top \mathcal{L}^2_{g_i} h_i(x_i) u_i \\
        & 
        \quad 
        + \big(\mathcal{L}_{f_i} \mathcal{L}_{g_i}h_i(\edit{\mathbf{x}_i})^\top +  \mathcal{L}_{g_i} \mathcal{L}_{f_i}h_i(\edit{\mathbf{x}_i})\big)u_i \edit{+ \mathcal{L}_{g_i}h_i(x_i)d(u_i)} 
        \\ 
        & 
        \quad 
        + \dot{\eta}_i\big(h_i(x_i),\edit{\mathbf{x}_i},u_i \big) + \kappa_i\big(\dot{h}_i(\edit{\mathbf{x}_i},u_i) + \eta_i(h_i(x_i)) \big).
    \end{aligned}
\end{equation}

\noindent
In this form, we may consider $a_{ij}(\edit{\mathbf{x}_i}) \in \mathbb{R}^{M_j}$ to be the effect that node $j \in \mathcal{N}_i^+$ has on the safety of node~$i \in [n]$ via its own control inputs $u_j$, and $c_i(\edit{\mathbf{x}_i^+},u_i)$ is the effect that node~$i \in [n]$ has on its own safety combined with the uncontrolled system dynamics. 
The functions in \eqref{eq:MO_funcs} in turn define the constraint sets
\begin{equation} \label{eq:MO_sets}
    \begin{aligned}
        \mathcal{C}_i^1 &= \{\edit{\mathbf{x}_i \in  \mathbb{R}^{N_i+N_i^+}}: \psi_i^0(x_i) \geq 0 \} \\
        \mathcal{C}_i^2 &= \{\edit{\mathbf{x}_i \in  \mathbb{R}^{N_i+N_i^+}}: \exists u_i \in \mathcal{U}_i \text{ s.t. } \psi_i^1(\edit{\mathbf{x}_i}, u_i) \geq 0 \}.
    \end{aligned}
\end{equation}

We may now formally define a collaborative control barrier function for node $i$ that takes into account the control actions of its incoming neighbors $j \in \mathcal{N}_i^+$.

\begin{definition}
    Let $\mathcal{C}_i^1$ and $\mathcal{C}_i^2$ be defined by \eqref{eq:MO_funcs} and \eqref{eq:MO_sets}, \edit{under Assumption~\ref{assume:u_dot_func_u}}. We have that $h_i$ is a \textit{collaborative control barrier function} (CCBF) for node~$i \in [n]$ if $h_i \in C^2$ and $\forall \edit{\mathbf{x}_i} \in \mathcal{C}_i^1 \cap \mathcal{C}_i^2$ there exists $(u_i, u_{\mathcal{N}_i^+}) \in \mathcal{U}_i \times \mathcal{U}_{\mathcal{N}_i^+}$ such that 
    \begin{equation} \label{eq:cNCBF_cond}
        \psi_i^2(\edit{\mathbf{x}_i^+}, u_i, u_{\mathcal{N}_i^+}) \geq 0,
    \end{equation}
    where $\eta_i, \kappa_i$ are class-$\mathcal{K}$ functions and $\eta_i \in D^r$, with $r \geq 1$.
\end{definition}

This definition yields the corresponding result for the forward invariance of the constraint sets in \eqref{eq:MO_sets}.
\begin{theorem}\label{thm:cNCBF}
    Given a networked dynamic system defined by \eqref{eq:sys_dyn_net} and constraint sets defined by \eqref{eq:MO_funcs} and \eqref{eq:MO_sets}, \edit{under Assumption~\ref{assume:u_dot_func_u}}, if $h_i$ is a CCBF then $\mathcal{C}_i^1 \cap \mathcal{C}_i^2$ is forward invariant. 
\end{theorem}
\begin{proof}
    If $h_i$ is a CCBF, then, \edit{by Assumption~\ref{assume:u_dot_func_u},} $\exists(u_i, u_{\mathcal{N}_i^+}) \in \mathcal{U}_i \times \mathcal{U}_{\mathcal{N}_i^+}$ such that \eqref{eq:cNCBF_cond} holds, i.e., we have that there exists a class-$\mathcal{K}$ function $\kappa_i$ such that
    \begin{equation*}
        \dot{\psi}_i^1(\edit{\mathbf{x}_i^+},u_i,u_{\mathcal{N}_i^+}) + \kappa_i(\psi_i^1(\edit{\mathbf{x}_i}, u_i)) \geq 0, \forall \edit{\mathbf{x}_i} \in \mathcal{C}_i^1 \cap \mathcal{C}_i^2.
    \end{equation*}
    Since $u_i$ appears in both $\psi_i^2(\edit{\mathbf{x}_i^+}, u_i, u_{\mathcal{N}_i^+})$ and $\psi_i^1(\edit{\mathbf{x}_i}, u_i)$, we must show that if $\edit{\mathbf{x}_i} \in \mathcal{C}_i^1 \cap \mathcal{C}_i^2$ and $\psi_i^2(\edit{\mathbf{x}_i^+}, u_i, u_{\mathcal{N}_i^+}) \geq 0$ for some $u_i \in \mathcal{U}_i$, then $\psi_i^1(\edit{\mathbf{x}_i}, u_i) \geq 0$ also. If \eqref{eq:cNCBF_cond} holds for all $\edit{\mathbf{x}_i} \in \mathcal{C}_i^1 \cap \mathcal{C}_i^2$, then $\forall x_i, x_{\mathcal{N}_i^+}, u_i \in \mathbb{R}^{N_i}\times \mathbb{R}^{N_i^+}\times \mathbb{R}^{M_i}$ where $\psi_i^1(\edit{\mathbf{x}_i}, u_i) = 0$, there exists $u_{\mathcal{N}_i^+} \in \mathbb{R}^{M_i^+}$ such that $\dot{\psi}_i^1(\edit{\mathbf{x}_i^+},u_i,u_{\mathcal{N}_i^+}) \geq 0$. Thus, we have $\psi_i^1(\edit{\mathbf{x}_i}, u_i) \geq 0, \forall \edit{\mathbf{x}_i} \in \mathcal{C}_i^1 \cap \mathcal{C}_i^2$, which directly implies $\psi_i^0(x_i) \geq 0, \forall \edit{\mathbf{x}_i} \in \mathcal{C}_i^1 \cap \mathcal{C}_i^2$. Therefore, we have that $\mathcal{C}_i^1 \cap \mathcal{C}_i^2$ is forward invariant.
\end{proof}

Notice, however, that the definition of the CCBF is still centered around node $i$ and its incoming 1-hop neighbors. In order to design a \edit{decentralized} control scheme for each node, we need to account for both the needs of node $i$ with respect to its incoming neighbors $j \in \mathcal{N}_i^+$ and the needs of its outgoing neighbors $k \in \mathcal{N}_i^-$ with respect to control actions made by node $i$. Thus, we propose a \edit{decentralized} collaborative control algorithm in the following section that ensures safety throughout a given networked dynamic system of the form in \eqref{eq:sys_dyn_net} via rounds of communication between neighboring nodes.

\section{Collaborative Safety} \label{sec:colab_safety}
In this section, we construct an algorithm that exploits the properties of the CCBF to communicate and process safety requests to and from neighbors, respectively.  
Typically, safety-critical control aims to minimally alter nominal control commands such that the controlled system is always safe. For example, given some nominal command control policies $u_i^{n}: \mathbb{R}^N \rightarrow \mathcal{U}_i$ and $u_{\mathcal{N}_i^+}^{n}: \mathbb{R}^N \rightarrow \mathcal{U}_{\mathcal{N}_i^+}$ for node~$i \in [n]$ and its neighbors $j \in \mathcal{N}_i^+$, respectively, as well as a valid cNCBF $h_i$, our safe control $u_i^s, u_{\mathcal{N}_i^+}^s$ may be computed as
\begin{equation} \label{eq:min_safe_control_colab}
    \begin{aligned}
        \argmin_{u_i, u_{\mathcal{N}_i^+} \in \mathcal{U}_i \times \mathcal{U}_{\mathcal{N}_i^+}} \quad & {\Vert u_i - u_i^{n} \Vert}^2 +{\Vert u_{\mathcal{N}_i^+} - u_{\mathcal{N}_i^+}^{n} \Vert}^2 \\
        \text{s.t.} \quad & \psi_i^2(\edit{\mathbf{x}_i^+}, u_i, u_{\mathcal{N}_i^+}) \geq 0.
    \end{aligned}
\end{equation}
However, ensuring that $h_i$ is a valid CCBF in practice may be infeasible since \eqref{eq:min_safe_control_colab} does not account for the effects of the incoming neighbors $\mathcal{N}_j^+, \forall j \in \mathcal{N}_i^+$. Therefore, in order to consider the needs of nodes in $\mathcal{N}_i^-$, we propose Algorithm~\ref{alg:colab_safety}, which uses rounds of communication to request shared responsibility for safety among incoming neighbors $\mathcal{N}_i^+$ while receiving requests from outgoing neighbors to determine feasible control constraints for node~$i \in [n]$.

\subsection{Algorithm Construction}
We now discuss the construction of Algorithm~\ref{alg:colab_safety}\edit{, its subroutines,} and the properties of its convergence to viable sets of constrained control actions that assist in guaranteeing neighbor safety.
The central idea of Algorithm~\ref{alg:colab_safety} involves rounds of \edit{collaboration} between nodes, where each round of \edit{collaboration} between nodes, centered on a node $i\in [n]$, involves the following steps:
\begin{enumerate}
    \item Receive (Send) requests from (to) neighboring nodes in $\mathcal{N}_i^-$ $\left( \mathcal{N}_i^+ \right)$
    \item \edit{Coordinate} requests and determine needed compromises for nodes in $\mathcal{N}_i^-$
    \item Send (Receive) adjustments to (from) neighboring nodes in $\mathcal{N}_i^-$ $\left( \mathcal{N}_i^+ \right)$.
\end{enumerate}
By \eqref{eq:psi2_grouped_terms} and \eqref{eq:cNCBF_cond}, the condition that node~$i \in [n]$ must satisfy to guarantee safety for a given \edit{2-hop neighborhood state} is
\begin{equation}
    \sum_{j \in \mathcal{N}_i^+} a_{ij}(\edit{\mathbf{x}_i}) u_j  +  c_i(\edit{\mathbf{x}_i^+}, u_i) \geq 0.
\end{equation}

\noindent
Thus, \edit{under Assumption~\ref{assume:u_dot_func_u},} we compute the maximum capability of node~$i \in [n]$ to achieve its safety constraints
\begin{equation}
\label{eq:max_cap}
    \bar{c}_i = \max_{u_i \in \mathcal{U}_i} c_i(\edit{\mathbf{x}_i^+}, u_i),
\end{equation}

\noindent
where we may interpret $\bar{c}_i$ as the total responsibility of node~$i \in [n]$, with $\bar{c}_i>0$ indicating a surplus and $\bar{c}_i<0$ a deficit of control capability at node $i\in [n]$. We then partition control responsibility among incoming neighbors $\bar{c}_{ij}, \forall j \in \mathcal{N}_i^+$ such that $\bar{c}_i = \sum_{j \in \mathcal{N}_i^+} \bar{c}_{ij}$, where each incoming neighbor is responsible for satisfying
\begin{equation}\label{eq:neighbor_responsibility}
    a_{ij}(\edit{\mathbf{x}_i}) u_j + \bar{c}_{ij} \geq 0.
\end{equation}

\noindent
Since there are an infinite \edit{number of} partitions of $\bar{c}_i$, we must choose a method for dividing responsibility among each neighbor $j \in \mathcal{N}_i^+$ for $i \in [n]$
\begin{equation} \label{eq:request_c}
    \bar{c}_{ij} = \frac{\bar{c}_{i} w_{ij}}{\sum_{j \in \mathcal{N}_i^+} w_{ij}},
\end{equation}
where $w_{ij}$ is given by a weighting function $W_i: \mathcal{N}_i^+ \rightarrow \mathbb{R}_{\geq 0}$ that determines how much responsibility node $i$ requests from its available neighbors. An example of a weighting function candidate is
\begin{equation} \label{eq:W_i}
    W_{i}(j) = |a_{ij}(\edit{\mathbf{x}_i})|; \forall i \in [n], \forall j \in \mathcal{N}_i^+,
\end{equation}
which allocates responsibility to neighbors $\mathcal{N}_i^+$ based on the magnitude effect (i.e. the 1-norm of $a_{ij}(\edit{\mathbf{x}_i})$) of their actions on node $i$.
Therefore, if $\exists u_j \in \mathcal{U}_j$ such that \eqref{eq:neighbor_responsibility} is satisfied for all $j \in \mathcal{N}_i^+$ for all $\edit{\mathbf{x}_i} \in \mathcal{C}_i^1 \cap \mathcal{C}_i^2$, then $h_i$ is a CCBF and $\overline{\mathcal{U}}_j = \{ u_j \in \mathcal{U}_j: \eqref{eq:neighbor_responsibility} \}$ is a viable set of control inputs for node $j \in \mathcal{N}_i^+$ that will keep node $i\in [n]$ safe.

However, \eqref{eq:neighbor_responsibility} may not always be feasible for all $j \in \mathcal{N}_i^+$, which would require an adjustment of the responsibility allocated to node $j$. To illustrate the process for updating the responsibility, we shift our perspective back to node~$i \in [n]$ and consider requests $\bar{c}_{ki}$ received from $k \in \mathcal{N}_i^-$. In order to satisfy the requests of outgoing neighbors, node $i$ must ensure that
\begin{equation}
    \bigwedge_{k \in \mathcal{N}_i^-} \big(a_{ki}(\edit{\mathbf{x}_k}) u_i + \bar{c}_{ki} \geq 0 \big)
\end{equation}
holds, given $u_i \in \mathcal{U}_i$. Thus, we may compute constraints on $u_i$ according to a series of linear equations which define the super-level sets
\begin{equation}\label{eq:in_neighbor_constraint_k}
    \overline{\mathcal{U}}_{ki} = \{ u_i \in \mathbb{R}^{M_i}: a_{ki}(\edit{\mathbf{x}_k}) u_i + \bar{c}_{ki} \geq 0 \}.
\end{equation}
Note that \eqref{eq:in_neighbor_constraint_k} does not yet constrain $u_i$ by $\mathcal{U}_i$, rather it is the set of all possible control actions that satisfy the request of node $k \in \mathcal{N}_i^-$. We then take the intersection of all requested constraints
\begin{equation}\label{eq:in_neighbor_constraint_all}
    \overline{\mathcal{U}}_{\mathcal{N}_i^-} = \bigcap_{k \in \mathcal{N}_i^-} \overline{\mathcal{U}}_{ki}.
\end{equation}
If $\mathcal{U}_i \cap \overline{\mathcal{U}}_{\mathcal{N}_i^-}$ is nonempty, then $\overline{\mathcal{U}}_i = \mathcal{U}_i \cap \overline{\mathcal{U}}_{\mathcal{N}_i^-}$ is a valid constrained control set and no adjustments need to be sent. Otherwise, if  $\mathcal{U}_i \cap \overline{\mathcal{U}}_{\mathcal{N}_i^-} = \emptyset$ and $\overline{\mathcal{U}}_{\mathcal{N}_i^-} \neq \emptyset$, then we may compute the action $\overline{u}_i \in \partial \mathcal{U}_i$ that yields the minimum distance between the two hulls $\mathcal{U}_i$ and $\overline{\mathcal{U}}_{\mathcal{N}_i^-}$ \citep{kaown2009fast}, respectively, and use this point to compute needed adjustments for neighbors in $\mathcal{N}_i^-$ where if
\begin{equation} \label{eq:adjust_request}
    a_{ki}(\edit{\mathbf{x}_k})\overline{u}_i + \bar{c}_{ki} < 0,
\end{equation}
then we must request an adjustment $\varepsilon_{ki} > 0$ such that
\begin{equation} \label{eq:adjsut_request_needed}
    a_{ki}(\edit{\mathbf{x}_k})\overline{u}_i + \bar{c}_{ki} + \varepsilon_{ki} = 0,
\end{equation}
and send this adjustment back to each node $k \in \mathcal{N}_i^-$ where \eqref{eq:adjust_request} is true.

Shifting our perspective again to consider adjustments $\varepsilon_{ij}$ requested from incoming neighbors $\mathcal{N}_i^+$, we must update the original request $\bar{c}_{ij} \gets \bar{c}_{ij} + \varepsilon_{ij}$ for all neighbors and compute the deficit in control responsibility
\begin{equation} \label{eq:control_deficit}
    \delta_i = \bar{c}_i - \sum_{j \in \mathcal{N}_i^+}\bar{c}_{ij},
\end{equation}
which we can use to make another round of requests to all neighbors that are not yet constrained. We repeat this process until either no adjustments are needed, or until all neighbors, $j \in \mathcal{N}_i^+$ are constrained.

Finally, we check our current capability $\bar{c}_i =  \max_{u_i \in \overline{\mathcal{U}}_i} c_i(\edit{\mathbf{x}_i^+}, u_i)$ with our now potentially constrained control set $\overline{\mathcal{U}}_i \subseteq \mathcal{U}_i$ and compute the deficit in control responsibility again through \eqref{eq:control_deficit}. If $\delta_i < 0$, then we must repeat the \edit{collaboration} process again to try and distribute responsibility amongst neighbors, otherwise, Algorithm~\ref{alg:colab_safety} will halt with a feasibly constrained control set $\overline{\mathcal{U}}_i$. 

\begin{algorithm}
\caption{\edit{Collaborative Safety}}\label{alg:colab_safety}
    \begin{algorithmic}[1]
        \Initialize{
            $\bar{c}_{ij} \gets 0, \forall j \in \mathcal{N}_i^+; \bar{c}_{ki} \gets 0, \forall k \in \mathcal{N}_i^-$ \\
            $i \gets i_0; \overline{\mathcal{U}}_i \gets \mathcal{U}_i; \tau_i \gets 0$
        }
        \Repeat
            \State $\tau_i \gets \tau_i + 1$
            \State $\bar{c}_i \gets \max_{u_i \in \overline{\mathcal{U}}_i} c_i(\mathbf{x}_i^+, u_i)$
            \State $\delta_i,\, \overline{\mathcal{U}}_i, \{\bar{c}_{ij}\}_{j \in \mathcal{N}_i^+}, \{\bar{c}_{ki}\}_{k \in \mathcal{N}_i^-}$ 
            $\gets$ Collaborate$\left(\bar{c}_i, \overline{\mathcal{U}}_i, \{\bar{c}_{ij}\}_{j \in \mathcal{N}_i^+}, \{\bar{c}_{ki}\}_{k \in \mathcal{N}_i^-}\right)$
            \Until{$\delta_i \geq 0$}
        \State\Return $\overline{\mathcal{U}}_i$
    \end{algorithmic}
\end{algorithm}

\begin{algorithm}
\caption{\edit{Collaborate}}\label{alg:colaborate}
    \begin{algorithmic}[1]
        \Initialize{
            $i \gets i_0;\; W_i \gets \eqref{eq:W_i};\; \overline{\mathcal{N}}_i^+\gets \emptyset$
        }
        \Inputs{
            \edit{
                $\bar{c}_i,\, \overline{\mathcal{U}}_i, \{\bar{c}_{ij}\}_{j \in \mathcal{N}_i^+}, \{\bar{c}_{ki}\}_{k \in \mathcal{N}_i^-}$
            }
        }
        \Repeat
            \State $\delta_i \gets \bar{c}_i - \sum_{j \in \mathcal{N}_i^+}\bar{c}_{ij}$
            \edit{\State $\{\delta_{ij}\}_{j\in \mathcal{N}_i^+} \gets \left\{ \frac{\delta_i w_{ij}}{\sum_{l \in \mathcal{N}_i^+ \setminus \overline{\mathcal{N}}_i^+} w_{il}}\right\}_{j \in \mathcal{N}_i^+}$}
            \State SEND to each $j \in \mathcal{N}_i^+ \setminus \overline{\mathcal{N}}_i^+ : \delta_{ij}$
            \State RECEIVE from all $k \in \mathcal{N}_i^- : \edit{\{\delta_{ki}\}_{k \in \mathcal{N}_i^-}}$
            \State {$\overline{\mathcal{U}}_i, \edit{\{\bar{c}_{ki}, \varepsilon_{ki}\}_{k \in \mathcal{N}_i^-}} \gets$ Coordinate$\left(\{\bar{c}_{ki}, \delta_{ki}\}_{k \in \mathcal{N}_i^-}\right)$}
            \State SEND to each $k \in \mathcal{N}_i^- : \varepsilon_{ki}$
            \State RECEIVE from all $j \in \mathcal{N}_i^+: \edit{ \{\varepsilon_{ij}\}_{j \in \mathcal{N}_i^+}}$
            \For{$j \in \mathcal{N}_i^+$}
                \State $\bar{c}_{ij} \gets \bar{c}_{ij} + \delta_{ij} + \varepsilon_{ij}$ 
                \If{$\varepsilon_{ij} > 0$}
                    \State $\overline{\mathcal{N}}_i^+ \gets \overline{\mathcal{N}}_i^+ \cup \{ j\}$
                \EndIf
            \EndFor
        \Until{{\footnotesize $(\overline{\mathcal{N}}_i^+ = \mathcal{N}_i^+) \lor (\varepsilon_{ij} = 0, \forall j \in \mathcal{N}_i^+ \land \varepsilon_{ki} = 0, \forall k \in \mathcal{N}_i^-)$}}
        \State \Return $\delta_i,\, \overline{\mathcal{U}}_i, \{\bar{c}_{ij}\}_{j \in \mathcal{N}_i^+}, \{\bar{c}_{ki}\}_{k \in \mathcal{N}_i^-}$
    \end{algorithmic}
\end{algorithm}

\begin{algorithm}
\caption{\edit{Coordinate}}\label{alg:coordinate}
    \begin{algorithmic}[1]
        \Initialize{
            \edit{$i \gets i_0;\; \varepsilon_{ki} \gets 0;\; a_{ki}(\mathbf{x}_k) \gets \eqref{eq:a_ij},  \forall k \in \mathcal{N}_i^-$
           }
        }
        \Inputs{
            \edit{$\{\bar{c}_{ki}, \delta_{ki}\}_{k \in \mathcal{N}_i^-}$
           }
        }

        \For{$k \in \mathcal{N}_i^-$}
           \State $\overline{\mathcal{U}}_{ki} \gets \{ u_i \in \mathbb{R}^{M_i}: a_{ki}(\mathbf{x}_k) u_i + \bar{c}_{ki} + \delta_{ki} \geq 0 \}$
        \EndFor

        \State $\overline{\mathcal{U}}_{\mathcal{N}_i^-} \gets \bigcap_{k \in \mathcal{N}_i^-} \overline{\mathcal{U}}_{ki}$

        \If{$\mathcal{U}_i \cap \overline{\mathcal{U}}_{\mathcal{N}_i^-} \neq \emptyset$ }
            \State $\overline{\mathcal{U}}_i \gets \mathcal{U}_i \cap \overline{\mathcal{U}}_{\mathcal{N}_i^-}$

        \Else
            \State $\overline{u}_i \gets$ getClosestPoint$(\mathcal{U}_i, \overline{\mathcal{U}}_{\mathcal{N}_i^-})$ \citep{kaown2009fast}
            
            \State $\overline{\mathcal{U}}_i \gets \{ \overline{u}_i \}$  
      
            \For{$k \in \mathcal{N}_i^-$}
                \If{$a_{ki}(\mathbf{x}_k)\overline{u}_i + \bar{c}_{ki} + \delta_{ki} < 0$}
                    \State $\varepsilon_{ki} \gets -(a_{ki}(\mathbf{x}_k)\overline{u}_i + \bar{c}_{ki} + \delta_{ki})$
                \EndIf
            \EndFor
        \EndIf
        \State $\bar{c}_{ki} \gets \bar{c}_{ki} + \delta_{ki} + \varepsilon_{ki}, \forall k \in \mathcal{N}_i^-$
        \State \Return $\overline{\mathcal{U}}_i, \edit{\{\bar{c}_{ki}, \varepsilon_{ki}\}_{k \in \mathcal{N}_i^-}}$
    \end{algorithmic}
\end{algorithm}

\subsection{\edit{Algorithm Convergence}}
\edit{
To assist in analyzing the convergence of Algorithm~\ref{alg:colab_safety}, we provide the following definition on the relationship of incoming requests from neighbors $k \in \mathcal{N}_i^-$ for all nodes $i \in [n]$.
}
\edit{
\begin{definition}\label{def:weakly_non_interfering}
     For a given node~$i \in [n]$, the set of neighbor constraints $h_{k}(x_k)$ for $k \in \mathcal{N}_i^-$ are said to be \textit{weakly non-interfering} if there exists a vector $a \in \mathbb{R}^{M_i}$ such that $a\cdot ~a_{ki}(\mathbf{x}_k) > 0, \forall k \in \mathcal{N}_i^-$, where $a_{ki}(\mathbf{x}_k)= \mathcal{L}_{g_i} \mathcal{L}_{f_k}  h_k(\mathbf{x}_k)$.
\end{definition}
}
\edit{
We borrow the terminology of non-interfering safety constraints from \citep{breeden2022compositions} which presents a method for composing multiple safety conditions for a single system with double integrator dynamics. We use the terminology of \textit{weakly non-interfering} constraints in contrast with \textit{non-interfering} constraints in \citep{breeden2022compositions} due to the more relaxed condition in Definition~\ref{def:weakly_non_interfering} which requires only that all vectors $a_{ki}(\mathbf{x}_k)$ lay in the same halfspace, rather than requiring $a_{ki}(\mathbf{x}_k) \cdot a_{li}(\mathbf{x}_l) \geq 0, \forall k,l \in \mathcal{N}_i^-$. One reason for this relaxation lies in the strength of guarantees obtained through constructing barrier functions a priori to be jointly feasible, as is proposed in \citep{breeden2022compositions}, versus treating safety constraints as defined independently for each agent according to their safety requirements and using a collaborative scheme to find jointly feasible safe actions at runtime.
}
\edit{
Using this definition of weakly non-interfering constraints, we have the following lemma.
\begin{lemma} \label{lem:non_empty_constraints}
    If the set of constraints $h_{k}(x_k)$ for all $k \in \mathcal{N}_i^-$ are weakly non-interfering, then $\overline{\mathcal{U}}_{\mathcal{N}_i^-} \neq \emptyset, \forall i \in [n]$.
\end{lemma}
\begin{proof}
    By Definition~\ref{def:weakly_non_interfering}, there must exist an open half-space in $\mathbb{R}^{M_i}$, defined by the vector $a \in \mathbb{R}^{M_i}$,
    \begin{equation*}
        \mathcal{U}_i^a = \{ u_i \in \mathbb{R}^{M_i}: a u_i > 0 \}
    \end{equation*}
    such that all vectors $a_{ki}(\mathbf{x}_k)$ for $k \in \mathcal{N}_i^-$ are contained in that halfspace. Thus, the convex hull created by the intersection of the halfspaces, defined by
    \begin{equation*}
        \overline{\mathcal{U}}_{ki} = \{ u_i \in \mathbb{R}^{M_i}: a_{ki}(\mathbf{x}_k) u_i + \bar{c}_{ki} + \delta_{ki} \geq 0 \},
    \end{equation*}
    must be nonempty in $\mathbb{R}^{M_i}$ for any values of $\bar{c}_{ki} + \delta_{ki}$.
\end{proof}
}
\edit{
In Figure~\ref{fig:weakly_non_interfering_proof_ex}, we illustrate an example of two jointly infeasible constraints for an agent $i\in [n]$ with $u_i \in \mathbb{R}^2$.
Additionally, we make the following assumption on the control constraints for node~$i \in [n]$.
\begin{assumption} \label{assume:U_i_convex_closed}
     For a given node $i\in [n]$, let $\mathcal{U}_i \subset \mathbb{R}^{M_i}$ be a nonempty, convex, closed set.
\end{assumption}
We use the following definition to describe pairs of neighbors that may cause certain safety requests to be infeasible for node~$i \in [n]$.
}
\edit{
\begin{definition}
    For a given node~$i \in [n]$, the neighbor constraint $h_{k}(x_k)$ is \textit{infeasible} if $\mathcal{U}_i \cap \overline{\mathcal{U}}_{ki} = \emptyset$. Further,
    the neighbor constraints $h_{k}(x_k)$ and $h_{l}(x_l)$ for $k,l \in \mathcal{N}_i^-$ are \textit{jointly infeasible} if $\mathcal{U}_i \cap \overline{\mathcal{U}}_{ki} \cap \overline{\mathcal{U}}_{li} = \emptyset$.
\end{definition}
}
\edit{
Infeasible and jointly infeasible constraints create the potential for no allowable control action from each agent in the system to satisfy all safety constraints.
}
\edit{
\begin{definition}
    We say that $x_1, \dots, x_n$ is a \textit{terminally infeasible state} if there does not exist a set of control inputs $u_1, \dots, u_n \in \mathcal{U}_1 \times \cdots \times \mathcal{U}_n$ that satisfies \eqref{eq:cNCBF_cond} for all $i \in [n]$.
\end{definition}
}
\edit{Using these definitions,} we have the following result on the convergence of Algorithm~\ref{alg:colab_safety}.
\begin{theorem} \label{thm:colab_alg_convergence}
    \edit{Let the set of constraints $h_{k}(x_k)$ for $k \in \mathcal{N}_i^-$ be weakly non-interfering and Assumptions~\ref{assume:u_dot_func_u} and \ref{assume:U_i_convex_closed} hold $\forall i \in [n]$. If $x_1, \dots, x_n$ is not a terminally infeasible state, then Algorithm~\ref{alg:colab_safety} will return at least one safe action for all nodes $i\in [n]$.} 
\end{theorem}
\begin{proof}
    First, note the algorithm time $\tau_i \in \mathbb{N}$ for node $i\in [n]$, which tracks the number of total iterations carried out by \edit{the repeat loop of Algorithm~\ref{alg:colab_safety}}. 
    We first step Algorithm~\ref{alg:colab_safety} through $\tau_i=1$, then show for $\tau_i >~1$ that $\bar{c}_{ij}, \forall j \in \mathcal{N}_i^+$ and $\bar{c}_{ki}, \forall k \in \mathcal{N}_i^-$ are decreasing and lower bounded. 
    
    At $\tau_i=1$, \edit{each node~$i\in [n]$ computes its maximum safety capability by \eqref{eq:max_cap}, which is well defined by Assumption~\ref{assume:u_dot_func_u}.}
    If $\bar{c}_i > 0$, then node $i$ communicates a control surplus $\delta_{ij} \geq 0, \forall j \in \mathcal{N}_i^+$ , otherwise it will communicate a control deficit of $\delta_{ij} < 0, \forall j \in \mathcal{N}_i^+$. Once requests have been sent and received, each node processes them using Algorithm~\ref{alg:coordinate}, which initializes responsibility surpluses and deficits across the network. 
    
    Note that if \edit{$\exists k \in \mathcal{N}_i^-$ such that} $\varepsilon_{ki}>0$, then
    \edit{
     there exists at least one pair of infeasible neighbors in $\mathcal{N}_i^-$ and Algorithm~\ref{alg:coordinate} will have selected the closest point $\overline{u}_i \in \partial \mathcal{U}_i$ between the two nonempty convex hulls $\mathcal{U}_i$ and $\overline{\mathcal{U}}_{\mathcal{N}_i^-}$, where by $\mathcal{U}_i$ is nonempty by Assumption~\ref{assume:U_i_convex_closed} and $\overline{\mathcal{U}}_{\mathcal{N}_i^-}$ is nonempty by Lemma~\ref{lem:non_empty_constraints}. Thus, the control input which node~$i \in [n]$ may choose that would be the most beneficial for a given neighbor $k \in \mathcal{N}_i^-$ will be the point $u_i \in \partial \mathcal{U}_i$ where the hyperplane 
     \begin{equation*}
         a_{ki}(x) u_i + \bar{c}_{ki}^* = 0 
     \end{equation*}
     first intersects $\partial \mathcal{U}_i$, with $\bar{c}_{ki}^*$ chosen such that the solution
     \begin{equation*}
         \overline{u}_{ki}^* = \argmax_{u \in \mathcal{U}_i} a_{ki}(x) u_i + \bar{c}_{ki}^*
     \end{equation*}
     is unique, where such a solution must always exist by Assumption~\ref{assume:U_i_convex_closed}. Therefore, for infeasible neighbors in $\mathcal{N}_i^-$, node $i$ will adjust each neighbor's request $\delta_{ki}$ such that they intersect exactly on $\partial \mathcal{U}_i$.
     }

     \edit{
     If adjusted neighbors $k \in \mathcal{N}_i^-$ send another request $\delta_{ki}$ to node $i$, by Algorithm~\ref{alg:colaborate} it will only be \textit{after} attempting to allocate the remainder of  $\delta_k$ to the remaining unconstrained neighbors in $\mathcal{N}_k^+$. Thus, if $\tau_i^0$ is the algorithm time which node $i$ sends an adjustment to a given neighbor $k \in \mathcal{N}_i^-$, then $\delta_{ki}^{\tau_i} \leq \delta_{ki}^{\tau_i + 1} \leq 0$ for all $\tau_i^0 \leq \tau_i$.
     }

     \edit{
     Any viable compromise-seeking action for all infeasible neighbors, if it exists, will be contained in the subspace
     \begin{equation}
         \overline{\partial \mathcal{U}}_i = \{ u_i \in \partial \mathcal{U}_i: a_{ki}^{\perp} u_i \geq 0, \forall k \in \mathcal{N}_i^- \},
     \end{equation}
     where $a_{ki}^{\perp}$ are the vectors orthogonal to $a_{ki}$ such that $a \cdot a_{ki}^{\perp} \geq 0$, with $a \in \mathbb{R}^{M_i}$ being any vector that satisfies the property of \textit{weakly non-interfering} for constraints $h_{k}(x_k)$ for $k \in \mathcal{N}_i^-$. 
     At each subsequent iteration of Algorithm~\ref{alg:colab_safety}, where $\tau > 1$, the set of potential comprise points for infeasible neighbors is given by
     \begin{equation}
          \overline{\partial \mathcal{U}}_i^{\tau} = \{u_i \in \partial \mathcal{U}_i: a_{ki} u_i  + \bar{c}_{ki}^\tau + \delta_{ki}^\tau  < 0 , \forall k \in \mathcal{N}_i^-\} \cap \overline{\partial \mathcal{U}}_i
     \end{equation}
     (an example of $\overline{\partial \mathcal{U}}_i$ and $\overline{\partial \mathcal{U}}_i^{\tau}$ for two neighbors where $u_i \in \mathbb{R}^2$ is shown in Figure~\ref{fig:weakly_non_interfering_proof_ex}).
     Thus, since $\delta_{ki}^{\tau_i} \leq \delta_{ki}^{\tau_i + 1} \leq 0$, as $\tau \rightarrow \infty$, and since $x_1, \dots, x_n$ is not terminally infeasible, the set $\overline{\underline{\partial \mathcal{U}}}_i^{\tau}$ must contract to at least one feasibly safe action on $\partial \mathcal{U}_i$. 
     }
\end{proof}

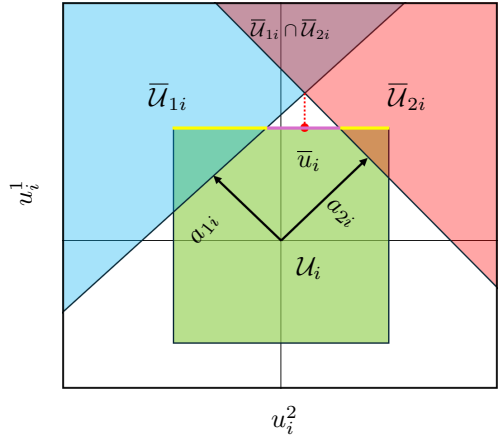
\begin{figure}
    \centering
    \begin{overpic}[width=.35\columnwidth]{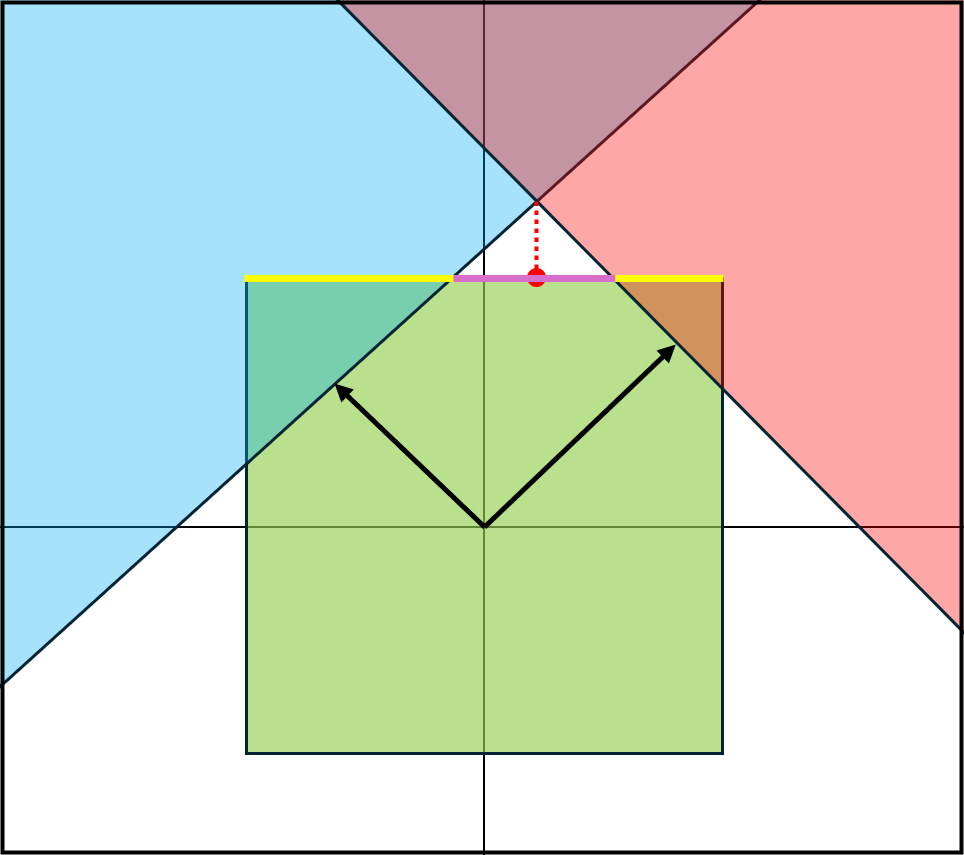}
    \put(-12,42){\parbox{0.75\linewidth}\normalsize \rotatebox{90}{$u_i^1$}}
    \put(48,-9){\parbox{0.75\linewidth}\normalsize{$u_i^2$}}
    \put(54,26){\parbox{0.75\linewidth}\normalsize{$\mathcal{U}_i$}}
    \put(29,33){\parbox{0.75\linewidth}\tiny \rotatebox{52}{$a_{1i}$}}
    \put(60,42){\parbox{0.75\linewidth}\tiny \rotatebox{-39}{$a_{2i}$}}
    \put(54,51){\parbox{0.75\linewidth}\normalsize{$\overline{u}_i$}}
    \put(75,65){\parbox{0.75\linewidth}\normalsize{$\overline{\mathcal{U}}_{2i}$}}
    \put(20,65){\parbox{0.75\linewidth}\normalsize{$\overline{\mathcal{U}}_{1i}$}}
    \put(43.4,81){\parbox{0.75\linewidth}\normalsize{\scalebox{.75}{$\overline{\mathcal{U}}_{1i} \cap \overline{\mathcal{U}}_{2i}$}}}
    \end{overpic}
    \vspace{4ex}
    \caption{
    \edit{
    An example of when node~$i \in [n]$ is constrained by two neighbors, where $u_i \in \mathbb{R}^2$. The constrained control space for agent~$i$, $\mathcal{U}_i$, is shaded green, with the feasibly safe control actions for neighbors 1 and 2 shaded in blue and red, respectively. Both $\overline{\mathcal{U}}_{1i}$ and $\overline{\mathcal{U}}_{2i}$ are individually feasible, but jointly infeasible, with the set of feasible, safe control actions for neighbors 1 and 2 shown in the purple-shaded region. The compromise-seeking action $\overline{u}_i \in \mathbb{R}^2$ is marked on the boundary of $\mathcal{U}_i$, which is the closest action in $\mathcal{U}_i$ to the feasible, safe control actions for both neighbors $\overline{\mathcal{U}}_{1i} \cap \overline{\mathcal{U}}_{2i}$. The set of all viable compromise-seeking actions $\overline{\partial \mathcal{U}}_i$ is shown by the yellow line and the current set of compromise-seeking actions $\overline{\partial \mathcal{U}}_i^\tau$ is shown by the pink line. 
    }
    } \label{fig:weakly_non_interfering_proof_ex}
\end{figure}

Since we are designing a \edit{decentralized} controller for each node, once safety needs have been communicated via Algorithm~\ref{alg:colab_safety} for each node~$i \in [n]$, control actions must then be selected independently of each other, yielding the safety-filtered control action for node $i$ computed as
\begin{equation} \label{eq:indpedendent_control}
    \begin{aligned}
        u_i^s(\mathbf{x}_i) = \argmin_{u_i \in \overline{\mathcal{U}}_i} \quad & {|| u_i - u_i^n(\mathbf{x}_i) ||}^2 \\
        \text{s.t.} \quad & \psi_i^1(\mathbf{x}_i, u_i) \geq 0,
    \end{aligned}
\end{equation}
where $ \overline{\mathcal{U}}_i$ is the constrained control set for node $i$ given by Algorithm~\ref{alg:colab_safety}. Thus, we yield the following \edit{decentralized} result for each node~$i \in [n]$.
\begin{corollary}
    If $\overline{\mathcal{U}}_i$ and $\overline{\mathcal{U}}_j$ computed by Algorithm~\ref{alg:colab_safety} are viable for $i \in [n]$ and $j \in \mathcal{N}_i^+$, respectively, for all $t \in \mathcal{T}$, then $\mathcal{C}_i^1 \cap \mathcal{C}_i^2$ is forward invariant under \eqref{eq:indpedendent_control}.
\end{corollary}
\begin{proof}
    For a given $\mathbf{x}_i \in \mathcal{C}_i^1 \cap \mathcal{C}_i^2$, if Algorithm~\ref{alg:colab_safety} returns with a viable $\overline{\mathcal{U}}_i$ then there exists $u_i, u_{\mathcal{N}_i^+} \in \mathcal{U}_i \times \mathcal{U}_{\mathcal{N}_i^+}$ such that
    \begin{equation} \label{eq:psi_2_cond}
        \psi_i^2(\mathbf{x}_i^+, u_i, u_{\mathcal{N}_i^+}) \geq 0.
    \end{equation} 
    Thus, by Theorem~\ref{thm:cNCBF}, if \eqref{eq:psi_2_cond} holds $\forall t \in \mathcal{T}$ then $\exists u_i \in \mathcal{U}_i$ such that $\psi_i^1(\mathbf{x}_i, u_i) \geq 0$. Therefore, $\mathcal{C}_i^1 \cap \mathcal{C}_i^2$ will be forward invariant $\forall t \in \mathcal{T}$ under \eqref{eq:indpedendent_control}.
\end{proof}

\section{Application: Networked SIS} \label{sec:applications}
To illustrate the theoretical results from Sections~\ref{sec:safe_w_local_control} and \ref{sec:colab_safety}, 
we apply our results to a networked susceptible-infected-susceptible (SIS) model for epidemic processes. While networked compartmental epidemic models can grow rapidly in complexity to model complex network interactions (such as multi-layered transportation networks \citep{vrabac2021capturing}) and disease behaviors (such as asymptomatic carriers and delayed symptoms \citep{butler2023optimal, zhang2023estimation, butler2021effect}), we choose a networked SIS model for the benefit of simplicity in demonstrating the results of this work. 
We define a networked system with $n$ nodes where $x_i$ is the proportion of infected population at node $i\in [n]$, with the infection dynamics defined by

\begin{equation} \label{eq:sis_networked}
    \dot{x}_i = -(\gamma_i+u_i) x_i + (1-x_i)\sum_{j \in [n]}\beta_{ij} x_j
\end{equation}

\noindent where $\gamma_i>0$ is the recovery rate at node $i$, $u_i \in \mathcal{U}_i \subset \mathbb{R}_{\geq 0}$ is a control input boosting the healing rate at node $i$, and $\beta_{ij} \geq 0$ is the networked connection going from node $j$ to node $i$ (note $\beta_{ii}$ is simply the infection rate occurring at node $i$). Note that we could also add a control term that reduces the infection rate $\beta_{ii}$ at each node $i$; however, we only consider one control term on the healing rate for simplicity. We also define the full state vector as $x \in [0,1]^n$ and the vector of all control inputs $u \in \mathbb{R}_{\geq 0}$. Let $\mathcal{U}_i = [0,\bar{u}_i]$, where $\bar{u}_i$ is the upper limit on boosting the healing rate at node $i$. We let each node define its individual safety constraint set as
\begin{equation} \label{eq:sis_indv_bar_func}
    h_i(x_i) = \bar{x}_i - x_i
\end{equation}
where $\bar{x}_i \in (0,1]$ is the defined safety threshold for the acceptable proportion of infected individuals at node $i$. Thus, the individual safe sets for each node are given by
\begin{equation}\label{eq:sis_indv_safe_set}
    \mathcal{C}_i = \{ x \in [0,1]^n: h_i(x_i) \geq 0 \}.
\end{equation}
\edit{Additionally, we select linear class-$\mathcal{K}$ functions $\eta_i(z) := \eta_i z$ and $\kappa_i(z) := \kappa_i z$, where $\eta_i, \kappa_i \in \mathbb{R}_{\geq 0}$, in the construction of the high-order barrier function safety conditions in \eqref{eq:MO_funcs}.}

We now demonstrate the performance of Algorithm~\ref{alg:colab_safety} in ensuring network safety for a simulated networked SIS epidemic process defined in \eqref{eq:sis_networked} using \edit{decentralized} control. For this example, we construct a simple 3-node system where $\beta_{ij} = 0.25$ and $\beta_{ii} = 0.5$ for all $j \neq i$ and $j = i$, respectively. To induce an endemic state, we set the healing rate $\gamma_i = 0.3$ for all nodes. 
In order to compute $\ddot{h}_i$ from \eqref{eq:sec_der_h_i_control_aff} using the networked SIS dynamics from \eqref{eq:sis_networked} we compute the following Lie derivative terms:
\begin{align*}
    \mathcal{L}_{f_i} h_i(x) &= - f_i(x) \\
    \mathcal{L}_{g_i} h_i(x) &= x_i \\
    \mathcal{L}_{f_i}^2 h_i(x) &= (\gamma_i - (1-2 x_i)\beta_{ii})f_i(x) \\
    \mathcal{L}_{g_i}^2 h_i(x) &= -x_i \\
    \mathcal{L}_{f_j}\mathcal{L}_{f_i}h_i(x) &= -(1-x_i)\beta_{ij}x_j f_j(x) \\
    \mathcal{L}_{g_j}\mathcal{L}_{f_i}h_i(x) &= (1-x_i)\beta_{ij}x_j^2 \\
    \mathcal{L}_{g_i}\mathcal{L}_{f_i}h_i(x) &= \left((1- 2 x_i)\beta_{ii}-\gamma_i - \sum_{j \in \mathcal{N}_i^+}\beta_{ij}x_j\right) x_i \\
    \mathcal{L}_{f_i}\mathcal{L}_{g_i}h_i(x) &= f_i(x).
\end{align*}

We set the safety constraints for each node as $\bar{x}_1 = 0.1$, $\bar{x}_2 = 0.12$, and $\bar{x}_3 = 0.18$, respectively, and constrain the control input to $\mathcal{U}_i = [0,0.75]$ for all $i \in [n]$. First, we test the performance of Algorithm~\ref{alg:colab_safety} in computing viable control constraints for each node given the safety needs of each neighbor. We simulate the system using Rung-Kutta numerical integration and at each time step of the simulation we have each node collaborate to communicate safety needs. Note, for \edit{any node $i\in [n]$, we have for any $x \in [0,1]^n$},
\begin{equation*}
    \mathcal{L}_{g_i}\mathcal{L}_{f_k}h_k(x) \cdot \mathcal{L}_{g_i}\mathcal{L}_{f_l}h_l(x) \geq 0, \forall k,l \in \mathcal{N}_i^-,
\end{equation*}
since \edit{$\mathcal{L}_{g_i}\mathcal{L}_{f_k}h_k(x) = (1-x_k)\beta_{ki}x_i^2 \geq 0$
for all $i \in [n]$ and $k \in \mathcal{N}_i^-$.} \edit{Thus, any set of requests from any neighbors will be weakly non-interfering by Definition~\ref{def:weakly_non_interfering}.}
Therefore, by Theorem~\ref{thm:colab_alg_convergence}, we have that Algorithm~\ref{alg:colab_safety} will always converge to either a feasibly constrained control set for each node, or at least one node will \edit{be in} a terminally infeasible state.

Once control constraints are computed via \edit{Algorithm~\ref{alg:colab_safety}}, each node selects a control action via \eqref{eq:indpedendent_control}. Using the initial condition of $x_0 = [0.04, 0.01, 0.02]$, we show in Figure~\ref{fig:eta_kappa_examples} the state trajectories and corresponding control inputs for systems.
\edit{
In this simulation, we see that node~1 reaches its healing rate upper bound as it approaches its maximum safe infection threshold; however, since node~3 has a higher infection capacity, it does not exert extra control until requested by node~1. Thus, with cooperation from node~3, node~1 can satisfy its safety requirements, whereas, without collaboration, node~1 would exceed its infection capacity. Note that it may be possible for no feasibly safe control action to exist if the requirements are too strict or the control budget $\mathcal{U}_i$ is too small, which could induce a terminally unfeasible state.  
}

\begin{figure}
    \centering
    \begin{subfigure}[b]{0.55\textwidth}
        \centering
        \includegraphics[width=\textwidth]{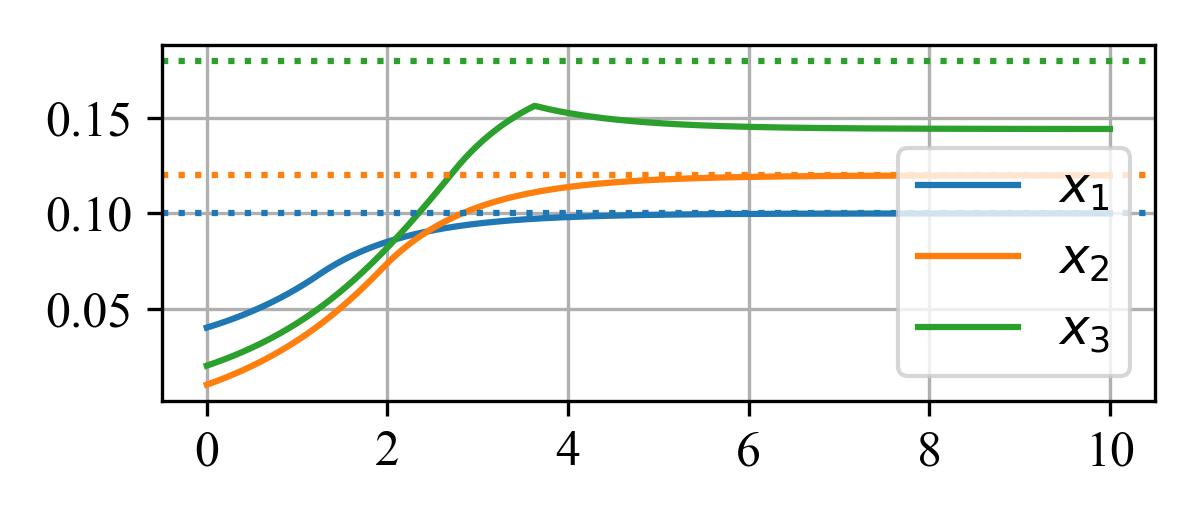}
    \end{subfigure}
    \begin{subfigure}[b]{0.55\textwidth}   
        \centering 
        \includegraphics[width=\textwidth]{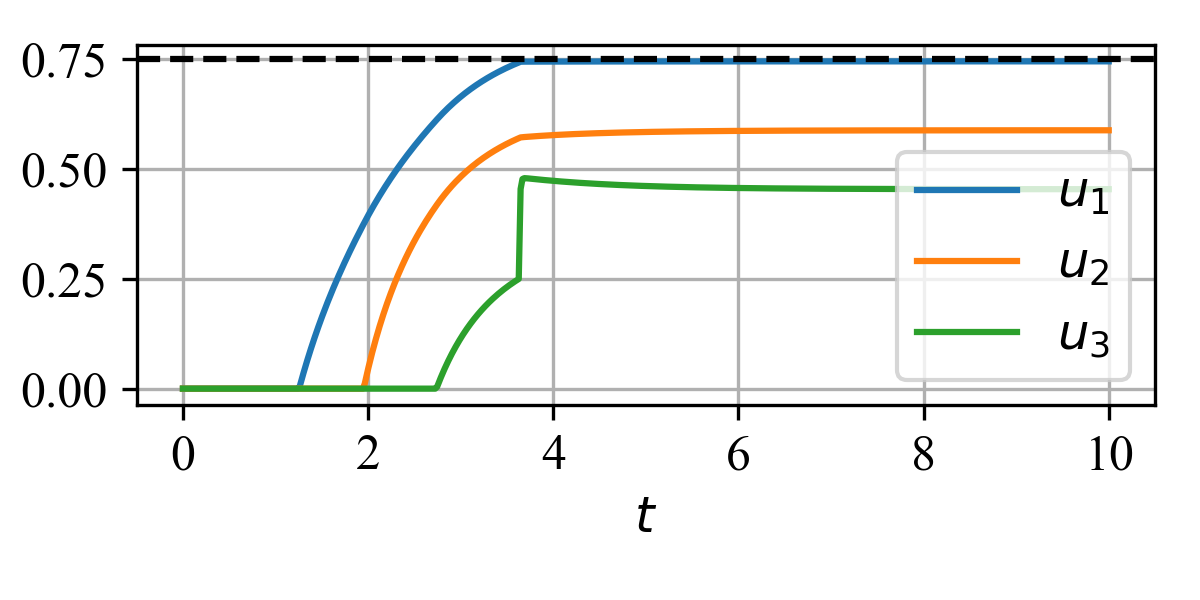}
    \end{subfigure}
    \vspace{-2ex}
    \caption{Simulated collaborative \edit{decentralized} controllers for a networked SIS model with $n=3$ where dotted lines represent each node's safety constraints $\bar{x}_i$ and the black dashed line showing the input constraint of $\mathcal{U}_i = [0, 0.75]$ for all nodes $i \in [n]$, respectively. 
    } 
    \label{fig:eta_kappa_examples}
\end{figure}

\section{Conclusion} \label{sec:conclusion}
In this work, we have presented methods for the \edit{decentralized}, collaborative safety-critical control of networked dynamic systems that exploit networked dynamics and structure to ensure individual safety goals are met. 
This novel approach leverages knowledge of networked dynamics in order to achieve stronger runtime assurances that tractably scale with network growth, which is of utmost importance as systems become more interconnected and prone to widespread failures.
While theory on safety-critical control of networks has been presented in this work, many directions of exploration and development remain open. 
Additional directions for future work include considering network influences beyond the 1-hop neighborhood as well as considerations of edge control in the coupling dynamics of networked models.

\bibliographystyle{unsrtnat}
\bibliography{references}  

\end{document}